\documentclass{amsart}
\usepackage{amsmath,amsfonts,amssymb,amsthm}
\usepackage{url}
\usepackage{graphicx,color}
\usepackage{enumerate}
\usepackage{amsthm}
\usepackage{amsmath}

\newtheorem{theorem}{Theorem}[section]
\newtheorem{lemma}[theorem]{Lemma}

\newtheorem*{theorem*}{Theorem}
\newtheorem{corollary}[theorem]{Corollary}
\theoremstyle{definition}
\newtheorem{definition}{Definition}[section]

\theoremstyle{remark}

\numberwithin{equation}{section}

\author{}
\address{}


\begin{document}

\title{Minimal displacement set  for weakly systolic complexes}

\author[Ioana-Claudia Laz\u{a}r]{
Ioana-Claudia Laz\u{a}r\\
Politehnica University of Timi\c{s}oara, Dept. of Mathematics,\\
Victoriei Square $2$, $300006$-Timi\c{s}oara, Romania\\
E-mail address: ioana.lazar@upt.ro}

\date{\today}

\begin{abstract}

We investigate the structure of the minimal displacement set in weakly systolic complexes. We show that such set is systolic and that it embeds isometrically into the complex. 
As corollaries, we prove that any isometry of a weakly systolic complex either fixes the barycentre of some simplex (elliptic case) or it stabilizes a thick geodesic (hyperbolic case).

\end{abstract}

\maketitle

\section{Introduction}

Curvature can be expressed both in metric and combinatorial terms. On the metric side, one can refer to nonpositively curved in the sense of Aleksandrov and Gromov,
i.e. by comparing small triangles in the space with triangles in the Euclidean plane. Such triangles must satisfy the CAT(0) inequality. On the combinatorial side, one can express curvature 
using a condition, called local $6$-largeness which was introduced independently by Chepoi \cite{Ch} (under the name of bridged complexes), Januszkiewicz-{\' S}wi{\c a}tkowski \cite{JS1} and Haglund \cite{Hag}.
In  \cite{BCCGO}, \cite{CCGHO}, \cite{CCHO}, \cite{Ch08}, \cite{ChOs}, \cite{HL}, \cite{L-8loc}, \cite{L-8loc2}, \cite{O-8loc} other conditions
of this type are studied.

Weakly systolic complexes were introduced in \cite{O-sdn} and further studied in \cite{ChOs}. Such complexes can be characterized as simply connected simplicial complexes satisfying some
local combinatorial conditions. This is analogous to 
CAT($0$) cubical complexes and systolic complexes. In graph-theoretical terms, the $1$-skeleta
of weakly systolic complexes (called weakly bridged graphs) satisfy the triangle
and quadrangle conditions, i.e., weakly bridged graphs are weakly modular. 

Properties of weakly systolic complexes
resemble very much the properties of spaces of non-positive curvature. We give a few examples.
CAT(0) simplicial complexes are collapsible (see \cite{AB}, Theorem $3.2.1$)). Both weakly systolic (see \cite{ChOs}, Corollary $4.3$) and systolic complexes (see \cite{lazar_2013}, Corollary $3.4$ or as a subclass of weakly systolic complexes) are also collapsible. Moreover, the fixed point theorem was studied for CAT(0) space (see \cite{BH}, chapter II.$2$, Corollary $2.8$), for systolic complexes (see \cite{Prz}, Theorem $1.2$), and for weakly systolic complexes (see \cite{ChOs}, Theorem $5.3$).

The purpose of the current paper is to investigate further similarities between the CAT(0), the systolic, and the weakly systolic worlds. Namely, we focus on the study of the minimal displacement set in a weakly systolic complex. This set was studied before for CAT(0) spaces (see \cite{BH}), for systolic complexes (see \cite{E2}) and for $8$-located complexes with the SD'-property (see \cite{L-8loc3}). We obtain the following main results.

\begin{theorem*}  (Theorem $3.2$) 
For a (simplicial) isometry $h$ of a weakly systolic complex $X$ having no fixed simplices,
the $1$-skeleton of the minimal displacement set ($\rm{Min}_{X}(h)$) is isometrically embedded into $X$.
\end{theorem*}

\begin{theorem*} (Theorem $3.4$) 
Let $h$ be a (simplicial) isometry of a weakly systolic complex $X$ having no fixed simplices. Then
the subcomplex $\rm{Min}_{X}(h)$ is systolic.
\end{theorem*}

As immediate consequence of these results, it follows that any isometry of a
weakly systolic complex either fixes the barycenter of some simplex (elliptic case) or it stabilizes a
thick geodesic (hyperbolic case).

\textbf{Acknowledgements}.
The author would like to thank Victor Chepoi for useful discussions.

\section{Preliminaries}

\subsection{Generalities}

Let $X$ be a simplicial complex.
We denote by $X^{(k)}$ the $k$-skeleton of $X, 0 \leq k < \dim X$.
A subcomplex $L$ in $X$ is called \emph{full} as a subcomplex of $X$ if any simplex of $X$ spanned by a set of vertices in $L$, is a simplex of $L$.
For a set $A = \{ v_{1}, ..., v_{k} \}$ of vertices of $X$, by $\langle A \rangle$ or by $\langle  v_{1}, ..., v_{k} \rangle$ we denote the \emph{span} of $A$, i.e. the smallest full subcomplex of $X$ that contains $A$.
We write $v \sim v'$ if $\langle  v,v' \rangle \in X$ (it can happen that $v = v'$).
We write $v \nsim v'$ if $\langle  v,v' \rangle \notin X$.
We call $X$ {\it flag} if any finite set of vertices which are pairwise connected by edges of $X$, spans a simplex of $X$.

We define the \emph{combinatorial metric} on the $0$-skeleton of $X$
as the number of edges in the shortest $1$-skeleton path joining two
given vertices.  We denote by $B_k(v)$ the ball of radius $k$ centered
at a vertex $v$.

A {\it cycle} ({\it loop}) $\gamma$ in $X$ is a subcomplex of $X$ isomorphic to a triangulation of $S^{1}$.
A \emph{full cycle} in $X$ is a cycle that is full as a subcomplex of $X$.
A $k$-\emph{wheel} in $X$ $(v_{0}; v_{1}, ..., v_{k})$ (where $v_{i}, i \in \{0,..., k\}$ are vertices of $X$) is a subcomplex of $X$ such that $\gamma = (v_{1}, ..., v_{k})$ is a full cycle and $v_{0} \sim v_{1}, ..., v_{k}$.
The \emph{length} of $\gamma$ (denoted by $|\gamma|$) is the number of edges of $\gamma$.
Let $g = (v_{1}, ..., v_{k})$ be a $1$-skeleton geodesic of $X$.
We call the \emph{length} of $g$ (denoted by $|g|$ or by $|(v_{1}, ..., v_{k})|$) the number of edges of $g$.

\subsection{Systolic complexes}

Let $\sigma$ be a simplex of $X$.  The \emph{link} of $X$ at $\sigma$,
denoted $X_{\sigma}$, is the subcomplex of $X$ consisting of all
simplices of $X$ which are disjoint from $\sigma$ and which, together
with $\sigma$, span a simplex of $X$.  We call a flag simplicial
complex $k$-\emph{large} if there are no full $j$-cycles in $X$, when
$4 \le j \le k-1$.  We say $X$ is \emph{locally} $k$-\emph{large} if
all its links are $k$-large.

We define the \emph{systole}
of $X$ to be
$\rm{sys}(X) = \min\{|\gamma| : \gamma$ is a full cycle in $X \}$.
Given a natural number $k \geq 4$, a simplicial complex X is:
\begin{enumerate}
\item $k$\emph{-large} if $\rm{sys}(X) \geq k$ and $\rm{sys}(X_{\sigma}) \geq k$ for each simplex $\sigma$ of $X$;
 
\item \emph{locally} $k$\emph{-large} if the star of every simplex of $X$ is $k$-large;

\item $k$\emph{-systolic} if it is connected, simply connected and  locally $k$-large.
\end{enumerate}

\subsection{Weakly systolic complexes}

By $\widehat{W}_{k} = (c; x_{1}, x_{2}, \dots, x_{k}; a)$ we denote a full $k$-wheel
$W_{k} = (c; x_{1}, x_{2}, \dots, x_{k})$ centered at $c$ plus a
triangle $\langle a, x_1, x_2 \rangle$ such that $a \neq c,$
$a \nsim c$ and $a$ is not adjacent to any other vertex of
$W_{k}$. We call
$\widehat{W}_{k}$ an
\emph{extended $k$-wheel}.  The $\widehat{W}_{5}$\emph{-condition}
states that for any $\widehat{W}_{5}$, there exists a vertex
$v \notin \widehat{W}_{5}$ such that $\widehat{W}_{5}$ is included in
$X_{v}$, i.e., $v$ is adjacent in $X^{(1)}$ to all vertices of
$\widehat{W}_{5}$.

\begin{definition}
Let $X$ be a flag simplicial complex and let $v$ be a vertex of
$X$. We say $X$ satisfies the $SD_{n}(v)$ \it{property} if for each
$i \leq n$ and each simplex $\sigma$ whose vertices are located in the
metric sphere $S_{i+1}(v)$, the set
$\sigma_{0} := X_{\sigma} \cap B_{i}(v)$ spans a non-empty simplex of
$X$.
\end{definition}

\begin{definition}
A \emph{weakly systolic complex} is a connected flag
simplicial complex $X$ which satisfies the $SD_{n}(v)$ property for all vertices $v \in X^{(0)}$ and for
all natural numbers $n$.
\end{definition}

%\begin{defn}
%A \emph{weakly bridged graph} is the $1$-skeleton
%of a weakly systolic complex.
%\end{defn}

\begin{definition}
 We say a graph $G$ is \emph{weakly modular} if its
distance function $d$ satisfies the following conditions:

$\bullet$ Triangle condition $(TC)$: for any three vertices $u, v, w$ of $G$ with $1 = d(v, w) < d(u, v) = d(u, w)$,
there exists a common neighbor $x$ of $v$ and $w$ such that $d(u, x) = d(u, v) - 1$.

$\bullet$ Quadrangle condition $(QC)$: for any four vertices $u, v, w, z$ of $G$ with $d(v, z) = d(w, z) = 1$ and
$2 = d(v, w) \leq d(u, v) = d(u, w) = d(u, z) - 1$, there exists a common neighbor $x$ of $v$ and $w$
such that $d(u, x) = d(u, v) - 1$.
\end{definition}

The following theorem gives a few characterizations of weakly systolic complexes.

\begin{theorem} 
For a connected flag simplicial complex $X$ the following
conditions are equivalent:
\begin{enumerate}
\item $X$ is weakly systolic;
\item $X^{(1)}$ is a weakly modular graph without full $4$-cycles;
\item $X$ is simply connected, it satisfies the $\widehat{W_{5}}$-condition, and it does not contain full $4$-cycles (see \cite{ChOs}, Theorem $3.1$).
\end{enumerate}
\end{theorem}

\subsection{Minimal displacement set}

Let $h$ be an isometry of a simplicial complex $X$. We define the \emph{displacement function}
$d_{h} : X^{(0)} \rightarrow \mathbf{N}$ by $d_{h}(x) = d_{X}(h(x),x)$. The \emph{translation length} of $h$ is
defined as $|h| = \min_{x \in X^{(0)}} d_{h}(x)$. 
If $h$ does not fix any simplex of $X$, then $h$ is called \emph{hyperbolic}. In such case one has
$|h| > 0$. Otherwise we call the isometry $h$ \emph{elliptic}. For a hyperbolic isometry $h$, we define the minimal displacement set $\rm{Min}_{X}(h)$
as the subcomplex of $X$ spanned by the set of vertices where $d_{h}$ attains its minimum.
Clearly $\rm{Min}_{X}(h)$ is invariant under the action of $h$.

For systolic complexes, the minimal displacement set is studied in \cite{E2}. Namely, it is proven that the following hold.

\begin{theorem}
Let $h$ be a hyperbolic isometry of a
systolic complex $X$. Then the subcomplex $\rm{Min}_{X}(h)$ is a systolic complex, isometrically
embedded into $X$ (see \cite{E2}, Propositions $3.3$ and $3.4$). 
\end{theorem}

%Let $h$ be an isometry of a simplicial complex $X$.
%An $h$-invariant geodesic in $X$ is called an \emph{axis} %of $h$. We say that $\rm{Min}_{X}(h)$ is
%the union of axes, if for every vertex $x \in %\rm{Min}_{X}%(h)$, there is an $h$-invariant geodesic
%passing through $x$, i.e. $\rm{Min}_{X}(h)$ can be written %as follows:
%\begin{center}
%$\rm{Min}_{X}(h) = \rm{span}\{\bigcup \gamma | \gamma$ is %an $h$-invariant geodesic$\}$ ($2.1$)
%\end{center}
%In this case, the isometry $h$ acts on $X$ as a %translation along the axes by the
%number $|h|$.

 Let $k \geq 2$ be an integer. Let  $A_{k}$ denote a simplicial complex with $A_{k}^{(0)} = \mathbf{Z}$ such that $\sigma \subset \mathbf{Z}$ spans a simplex if and only if $|a-a'| \leq k$ for all $a,a' \in \sigma$. 
 
 A \it{thick geodesic} in a simplicial complex $X$ is the full subcomplex $A_{k} \subset X$, $k \geq 1$ such that 
\begin{center}
$a-a' = j k , j \in \mathbf{Z} \implies d_{X}(a,a') = d_{A_{k}}(a,a')$.
\end{center}

%Any isometry $h$ of $A_{k}$ which is hyperbolic is of the %form $x \rightarrow x+n, \forall n \in \mathbf{Z}$.

%\begin{figure}[ht]
%    \centering
  %  \includegraphics[width=0.4\textwidth]{fig8.eps}
  %  \caption{}
    % \label{Case A, $|\gamma| = 2k$}
%\end{figure}

\section{Minimal displacement set for weakly systolic complexes}

In this section we study the structure of the minimal displacement set in a weakly systolic complex. 

\begin{lemma}\label{3.1}
Let $h$ be a simplicial isometry of a weakly systolic complex $X$ without fixed simplices. Choose
a vertex $v \in \rm{Min}_{X}(h)$ and a geodesic $\alpha \subset X^{(1)}$ joining $v$ with $h(v)$. Consider a simplicial path
$\gamma : \mathbf{R} \rightarrow X$ (where $\mathbf{R}$ is given a simplicial structure with $\mathbf{Z}$ as the set of vertices) being the
concatenation of geodesics $h^{n}(\alpha), n \in \mathbf{Z}$.
Then $\gamma$ is a $|h|$-geodesic (i.e. $d(\gamma(a), \gamma(b)) = |a-b|$
if $a, b$ are such integers that $|a - b| \leq |h|$). In particular, $\rm{Im}(\gamma) \subset \rm{Min}_{X}(h)$.
\end{lemma}

\begin{proof}
 The proof is similar to the one given in \cite{E2}, Fact $3.2$ for systolic complexes.
We prove the statement for $|a-b| = |h|$ (this implies the general case). Then, by the
construction of $\gamma$, either $\gamma(b) = h(\gamma(a))$ or $\gamma(a) = h(\gamma(b))$. Thus we have $d(\gamma(a), \gamma(b)) \geq |h|$.
The opposite inequality follows from the fact that $\gamma$ is a simplicial map.
\end{proof}

Below we show the paper's main result.

\begin{theorem}\label{3.2}
For a (simplicial) isometry $h$ of a weakly systolic complex $X$ having no fixed simplices,
the $1$-skeleton of $\rm{Min}_{X}(h)$ is isometrically embedded into $X$.
\end{theorem}

\begin{proof}

The construction is similar to the one given in \cite{E2}, Proposition $3.3$ for systolic complexes.

Suppose the $1$-skeleton of $\rm{Min}_{X}(h)$ is not isometrically embedded. Then there exist
vertices $v, w \in \rm{Min}_{X}(h)$ such that no geodesic in $X$ with endpoints $v$ and $w$ is contained in
$\rm{Min}_{X}(h)$. Choose $v$ and $w$ so that $d(v,w)$ is minimal (clearly, $d(v,w) > 1$). Join $v$ with
$h(v)$, $w$ with $h(w)$ and $v$ with $w$ by geodesics $\alpha, \beta$ and $\gamma$, respectively. Then $h(v)$ is joined
with $h(w)$ by $h(\gamma)$. Note that $h(\gamma)$ is also a geodesic. We have $|\alpha| = |\beta| = |h|$, $|\gamma| = |h(\gamma)| > 1$.

According to Lemma \ref{3.1}, we have
$\alpha, \beta \subset \rm{Min}_{X}(h)$.
Then, by minimality of $d(v, w)$, geodesics $\alpha$ and $\gamma$ intersect only at the endpoints. The same holds for the geodesics $\alpha$ and $h(\gamma)$, $\beta$ and $\gamma$, $\beta$ and $h(\gamma)$, respectively.
Suppose there is a vertex $u \in \gamma \cap h(\gamma)$. Then $h(u) \in h(\gamma)$ and $h(u) \neq u$, since $h$ has no
fixed simplices. We may assume, not losing generality, that $h(v), u, h(u)$ and $h(w)$ lie on $h(\gamma)$
in this order. Then
$d(u,h(u)) = d(h(v),h(u)) - d(h(v),u) = d(v,u) - d(h(v),u) \leq d(v,h(v)) = |h|$.
So $u \in \rm{Min}_{X}(h)$, contradicting the fact that no geodesic in $X$ with endpoints $v$ and $w$ is contained in $\rm{Min}_{X}(h)$.
Thus either the geodesics $\alpha, \beta, \gamma, h(\gamma)$ are pairwise disjoint but the endpoints or $\alpha$ and
$\beta$ have nonempty intersection. In both situations we proceed as follows. 

 Let $v',w' \in \gamma, v' \sim v, w' \sim w$.
Let $y,x \in \gamma$ be adjacent vertices such that $d(x,v) = d(y,v) + 1$. It may happen that $y = v$ or $x = w$ but not simultaneously due to the fact that $d(v,w) > 1$. The vertex $y$ is chosen such that it is the last vertex on $\gamma$ with $d(y,h(y)) = d(y,v) + d(v,h(v)) + d(h(v),h(y))$ (i.e.,
$y$ is the last vertex on $\gamma$ to be joined with $h(y)$ by the left of the cycle $\gamma \star \beta \star h(\gamma) \star \alpha$). The vertex $x$ is chosen such that it is the first vertex on $\gamma$ with $d(x,h(x)) = d(x,w) + d(w,h(w)) + d(h(w),h(x))$ (i.e., $x$ is the first vertex on $\gamma$ to be joined with $h(x)$ by the right of the cycle $\gamma \star \beta \star h(\gamma) \star \alpha$). If there exists $a \in \alpha$ such that $v \sim a \sim v'$, the geodesic from $y$ to $h(y)$ contains the edge $\langle a,v'\rangle$ (not the edges $\langle a,v\rangle$, $\langle v,v'\rangle$). In such case we still use the notation $\gamma \star \beta \star h(\gamma) \star \alpha$ although the cycle has one missing corner. The same convention holds if the cycle $\gamma \star \beta \star h(\gamma) \star \alpha$ has a missing corner on the left and another on the right.

%If there exists $b \in \alpha$ such that $h(v) \sim b \sim %h(v')$, the geodesic from $y$ to $h(y)$ contains only the edge %$\langle b,h(v')\rangle$ (not the edges $\langle %b,h(v)\rangle$, $\langle h(v),h(v')\rangle$). 

Assume $|\gamma| = 2$. Then either $y=v$ or $x=w$. Let $x=w$. The other case can be treated similarly. Suppose there are no vertices $a \in \alpha$, $c \in \beta$ such that $a \sim y \sim c$ and there are no vertices $b \in \alpha$, $e \in \beta$ such that $b \sim h(y) \sim e$. Then $d(v,h(y)) = d(w,h(y)) = |h| + 1$. Since $d(v,w) = 2$, by (QC), there exists a vertex $s \sim v,w$ such that $d(s,h(y)) = |h|$. Due to weak systolicity of $X$, a full $4$-cycle $(v,y,w,s)$ is forbidden. Then $y \sim s$. Note that $d(y,h(y)) = d(y,s) + d(s,h(y)) = |h|+1$. Because $d(y,h(y)) = |h|+2$, this implies a contradiction. Suppose there exists a vertex $b \in \alpha$ with $h(v) \sim b \sim h(y)$, but there do not exist vertices $a \in \alpha$; $c,e \in \beta$ with $v \sim a \sim y, w \sim c \sim y$, $h(w) \sim e \sim h(y)$.
Note that $d(b,w)=d(h(y),w)=|h|+1$. Then, by (TC), there exists a vertex $q \sim b,h(y)$ such that $d(q,w)=|h|$. Note that $d(q,w)=d(h(w),w)=|h|$. If $q \sim h(w)$, by (TC), there exists a vertex $r \sim q,h(w)$ such that $d(r,w)=|h|-1$. If $d(q,h(w))=2$, by (QC), there exists a vertex $r \sim q,h(w)$ such that $d(r,w)=|h|-1$. A full $4$-cycle $(q,r,h(w),h(y))$ is forbidden. If $h(y) \sim r$, then $d(h(y),w) = d(h(y),r)+d(r,w) = |h|.$ This yields contradiction with $d(h(y),w)=|h|+1$. Therefore $q \sim h(w)$. Let $l \in \alpha, b \sim l$, $d(b,v)=d(l,v)+1$. Then $d(l,w)=d(q,w)=|h|$. If $l \sim q$, by (TC), there exists a vertex $p \sim l,q$ such that $d(p,w)=|h|-1$. If $d(l,q)=2$, by (QC), there exists a vertex $p \sim l,q$ such that $d(p,w)=|h|-1$. A full $4$-cycle $(l,p,q,b)$ is forbidden. If $l \sim q$, we have $d(v,h(w))= d(v,l)+d(l,q)+d(q,h(w))=|h|$. This implies contradiction with $d(v,h(w))=|h|+1$. If $p \sim b$, we get $d(b,w)=d(b,p)+d(p,w)=|h|$. This yields contradiction with $d(b,w)=|h|+1$. 
Suppose there exists a vertex $a \in \alpha$ such that $v \sim a \sim y$ and there exists a vertex $b \in \alpha$ such that $h(v) \sim b \sim h(y)$. Then $d(y,h(y))=|h|$ implying $y \in \rm{Min}_{X}(h)$. This yields a contradiction.
For the rest of the proof, let $|\gamma| \geq 3$.

If there exist vertices $a,b \in \alpha$ such that $v' \sim a \sim v$, $h(v') \sim b \sim h(v)$, then $d(v', h(v')) = |h|$. This implies that $v' \in \rm{Min}_{X}(h)$ contradicting the fact that no geodesic in $X$ with endpoints $v$ and $w$ is contained in $\rm{Min}_{X}(h)$. We distinguish the following cases:

$\bullet$ $d(v',h(v')) = d(w',h(w')) = |h|+2$ (Case A);

$\bullet$ $d(v',h(v')) = |h|+1$, $d(w',h(w')) = |h|+2$ (Case B, Case C); 

$\bullet$ $d(v',h(v')) = |h|+2$, $d(w',h(w')) = |h|+1$ (Case D, Case E); 

$\bullet$ $d(v',h(v')) = d(w',h(w')) = |h|+1$ (Case F, Case G, Case H, Case I). We treat these cases below and we obtain in each case a contradiction. 

Case A. There do not exist vertices $a, b \in \alpha$; $c, e \in \beta$ such that $v \sim a \sim v'$, $h(v) \sim b \sim h(v')$, $w \sim c \sim w'$, $h(w) \sim e \sim h(w')$, respectively. Then $d(v',h(v')) = d(w',h(w')) = |h|+2$. Let $z \in \gamma$ such that $z \sim y$, $d(y,v)=d(z,v)+1$ (possibly with $v=z$).

Case A.1. Let $|\gamma| = 2k, k \geq 1$.

If $d(v,x)=d(x,w)=k$, we have $d(y,h(y)) = 2k-2+|h|$, $d(x,h(x))=2k+|h|$. If $d(v,y)=d(y,w)=k$, we have $d(y,h(y)) = 2k+|h|$, $d(x,h(x))=2k-2+|h|$. We treat only the case $d(v,y) = d(y,w) = k$. The other case can be treated similarly.\

\begin{figure}[ht]
    \centering
    \includegraphics[width=0.4\textwidth]{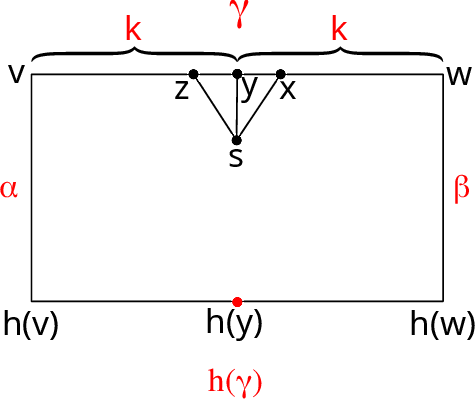}
    \caption{}
    %\label{Case A, $|\gamma| = 2k$}
\end{figure}

 Note that $d(x,h(y)) = d(z,h(y)) = 2k - 1 + |h|$. Then, since $d(z,x) = 2$, by (QC), there exists a vertex $s \sim x,z$ such that $d(s,h(y)) = 2k - 2 + |h|$. A full $4$-cycle $(z,y,x,s)$ is forbidden. Then $y \sim s$, and therefore $d(y,h(y)) = d(y,s) + d(s,h(y))= 2k - 1 + |h|$. This implies a contradiction with $d(y,h(y)) = 2k + |h|$.

Case A.2. Let $|\gamma| = 2k + 1, k \geq 1$.

Assume $d(v,y) = k+1$. Then $d(y,w) = k$. Note that $d(y,h(y)) = 2k+2+|h|$ by the left of the cycle $\gamma \star \beta \star h(\gamma) \star \alpha$, while $d(y,h(y)) = 2k+|h|$ by the right of the cycle $\gamma \star \beta \star h(\gamma) \star \alpha$. The point $y$ is chosen such that the geodesic from $y$ to $h(y)$ passes by the left of the cycle $\gamma \star \beta \star h(\gamma) \star \alpha$. The case $d(v,y) = k+1$ is therefore not possible. Hence $d(v,y)= k$.

Note that $d(x,h(x)) = d(z,h(x)) = 2k + |h|$. Then, since $d(z,x) = 2$, by (QC), there exists a vertex $s \sim x,z$ such that $d(s,h(x)) = 2k -1 + |h|$. A full $4$-cycle $(z,y,x,s)$ is forbidden.
Hence $y \sim s$. Hence $d(y,h(x))=d(y,s)+d(s,h(x))=2k+|h|$. This yields contradiction with $d(y,h(x))=2k+1+|h|$.

In conclusion $d(v',h(v')) \neq |h| + 2$. This completes case $A$.

Case $B.$ There exists a vertex $a \in \alpha$ such that $v \sim a \sim v'$.  There do not exist vertices $c \in \alpha$; $b,e \in \beta$ such that $h(v) \sim c \sim h(v')$, $w \sim b \sim w'$, $h(w) \sim e \sim h(w')$, respectively. Then $d(v',h(v')) = |h|+1$.

Case $B.1$. Let $|\gamma| = 2k, k \geq 1$.

 Suppose $d(v,x)=d(x,w)=k$. Note that $d(x,h(x))=2k-1+|h|$ by the left of the cycle $\gamma \ast \beta \ast h(\gamma) \ast \alpha$ while $d(x,h(x))=2k+|h|$ by the right of the cycle $\gamma \ast \beta \ast h(\gamma) \ast \alpha$. This yields contradiction with the fact that the vertex $x$ was chosen such that the geodesic from $x$ to $h(x)$ passes by the right of the cycle $\gamma \ast \beta \ast h(\gamma) \ast \alpha$. Therefore $d(v,y)=d(y,w)=k$. Let $z \in \gamma$ such that $z \sim y$, $d(y,v)=d(z,v)+1$ (possibly with $z = v$). Let $t \in \gamma$ such that $t \sim x$, $d(x,w)=d(t,w)+1$ (possibly with $t = w$).

Note that $d(y,h(y)) = d(x,h(y)) = 2k-1+|h|$. Then, by the (TC), there exists a vertex $s \sim x,y$ such that $d(s,h(y)) = 2k-2+|h|$. Note that $d(z,h(y)) = d(s,h(y)) = 2k-2+|h|$. If $d(z,s)=2$, by (QC), there exists a vertex $m \sim z,s$ such that $d(m,h(y))=2k-3+|h|$. A full $4$-cycle $(z,y,s,m)$ is forbidden. If $m \sim y$, we get $d(y,h(y))=d(y,m)+d(m,h(y))=2k-2+|h|$. This yields a contradiction with $d(y,h(y))=2k-1+|h|$. Then $z \sim s$. If $t \sim s$, we get $d(z,t)=2$ which yields contradiction with $\gamma$ being a geodesic. Hence $d(s,t)=2$. Note that $d(t,h(y)) = d(s,h(y)) = 2k-2+|h|$. By (QC), there exists a vertex $r \sim s,t$ such that $d(r,h(y))=2k-3+|h|$. The $4$-cycle $(x,s,r,t)$ is not allowed to be full. Hence $x \sim r$ and therefore $d(x,h(y)) = d(x,r) + d(r,h(y)) = 2k-2+|h|$. This yields contradiction with $d(x,h(y))=2k-1+|h|$.

Case $B.2$. Let $|\gamma| = 2k + 1, k \geq 1$. 

Suppose $d(v,y) = k+1$. Then $d(y,w) = k$. So $d(y,h(y)) = 2k + 1 + |h|$ by the left of the cycle $\gamma \star \beta \star h(\gamma) \star \alpha$, while $d(y,h(y)) = 2k + |h|$ by the right of the cycle $\gamma \star \beta \star h(\gamma) \star \alpha$. The point $y$ is chosen such that the geodesic from $y$ to $h(y)$ passes by the left of the cycle $\gamma \star \beta \star h(\gamma) \star \alpha$. So this case is not possible.
Then $d(v,y) = k, d(y,w) = k+1$. 

Let $z \in \gamma$ such that $y \sim z$, $d(y,v)=d(z,v)+1$ (possibly with $z=v$). Let $u \in \gamma$ such that $u \sim x$, $d(x,w)=d(u,w)+1$ (possibly with $u=w$). 
Note that $d(h(y),x) = d(h(x),x) = 2k + |h|$. Then (TC) implies that there exists a vertex $s \sim h(x),h(y)$ such that $d(s,x) = 2k - 1 + |h|$. 
 
Case $B.2.1$. Assume $s \sim h(z)$. 
If $h(u) \sim s$, we get contradiction with $h(\gamma)$ being a geodesic. Therefore $h(u) \nsim s$. Note that $d(s,x) = d(h(u),x) = 2k-1+|h|$. Since $d(h(u),s) = 2$, by (QC), there exists a vertex $t \sim s,h(u)$ such that $d(t,x) = 2k-2+|h|$. A full $4$-cycle $(s,h(x),h(u),t)$ is forbidden. So $t \sim h(x)$. Then $d(h(x),x) = d(h(x),t) + d(t,x) = 2k-1+|h|$. This implies contradiction with $d(h(x),x) =2k+|h|$.

\begin{figure}[ht]
    \centering
    \includegraphics[width=0.4\textwidth]{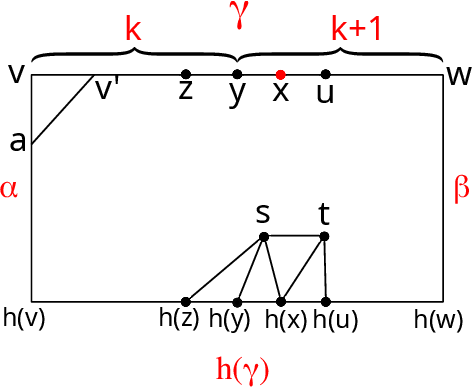}
    \caption{}
   % \label{Case B, $|\gamma| = 2k+1$, $h(z) \nsim s$}
\end{figure}

Case $B.2.2$. Assume $s \nsim h(z)$. Note that $d(s,x) = d(h(z),x) = 2k - 1 + |h|$. Since $d(s,h(z)) = 2$, by (QC), there exists a vertex $m \sim s,h(z)$ such that $d(m,x) = 2k - 2 + |h|$. A full $4$-cycle $(m,s,h(y),h(z))$ is forbidden. So $m \sim h(y)$. Then $d(h(y),x)=d(h(y),m)+d(m,x)=2k-1+|h|$. This implies contradiction with $d(h(y),x)=2k+|h|$.

\begin{figure}[ht]
    \centering
    \includegraphics[width=0.4\textwidth]{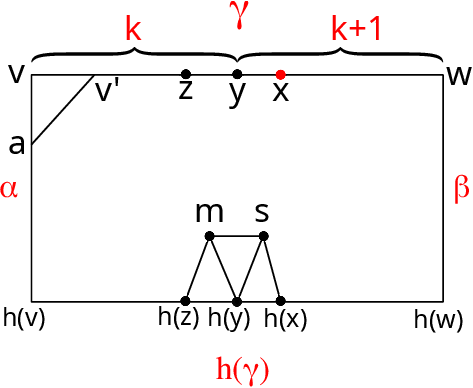}
    \caption{}
  %  \label{Case B, $|\gamma| = 2k+1$, $h(z) \sim s$}
\end{figure}

In conclusion we have $d(v',h(v')) \neq |h|+1$. This completes case $B.$

Case $C.$  There exists a vertex $b \in \alpha$ such that $h(v) \sim b \sim h(v')$. There do not exist vertices $a \in \alpha$; $c, e \in \beta$ such that $v \sim a \sim v'$, $w \sim c \sim w'$, $h(w) \sim e \sim h(w')$, respectively. This case can be treated similarly to case $B$ for $h(v)=v$, $h(w)=w$.

Case $D$. There exists a vertex $c \in \beta$ such that $w \sim c \sim w'$. There do not exist vertices $a,b \in \alpha$; $e \in \beta$ such that $v \sim a \sim v'$, $h(v) \sim b \sim h(v')$, $h(w) \sim e \sim h(w')$, respectively. This case is similar to case $B$ for $w=v$, $h(w)=h(v)$.

Case $E$. There exists a vertex $e \in \beta$ such that $h(w) \sim e \sim h(w')$. There do not exist vertices $a,b \in \alpha$; $c \in \beta$ such that $v \sim a \sim v'$, $h(v) \sim b \sim h(v')$, $w \sim c \sim w'$, respectively. This case is similar to case $B$ for $v=h(w)$, $w=h(v)$.

Case $F$. There exist vertices $a \in \alpha$, $c \in \beta$ such that $v \sim a \sim v'$, $w \sim c \sim w'$, respectively. There do not exist vertices $b \in \alpha$, $e \in \beta$ such that $h(v) \sim b \sim h(v')$, $h(w) \sim e \sim h(w')$, respectively. This case is similar to case $A$ for $v=v'$, $w=w'$, $h(v)=h(v')$, $h(w)=h(w')$.

Case $G$. There exist vertices $b \in \alpha$, $e \in \beta$ such that $h(v) \sim b \sim h(v')$, $h(w) \sim e \sim h(w')$, respectively. There do not exist vertices $a \in \alpha$, $c \in \beta$ such that $v \sim a \sim v'$, $w \sim c \sim w'$, respectively. This case is similar to case $A$ for $h(v)=h(v')$, $h(w)=h(w')$, $v=v'$, $w=w'$.

Case $H.$ There exist vertices $a \in \alpha$, $e \in \beta$ such that $v \sim a \sim v'$,  $h(w) \sim e \sim h(w')$, respectively. There do not exist vertices $c \in \beta$, $b \in \alpha$ such that $w \sim c \sim w'$, $h(v) \sim b \sim h(v')$, respectively.

Case $H.1$. Let $|\gamma| = 2k, k \geq 1$. 

If $d(v,y)=d(y,w)=k$, we get $d(y,h(y))=2k-1+|h|$, $d(x,h(x))=2k-3+|h|$. If $d(v,x)=d(x,w)=k$, we get $d(y,h(y))=2k-3+|h|$, $d(x,h(x))=2k-1+|h|$. 
We treat only the case $d(v,y)=d(y,w)=k$. The other case can be treated similarly. Let $z \in \gamma$ such that $z \sim y$, $d(y,v)=d(z,v)+1$ (possibly with $z=v$). Note that $d(z,h(y))=d(x,h(y))=2k-2+|h|$. Since $d(z,x)=2$, by (QC), there exists a vertex $s \sim x,z$ such that $d(s,h(y))=2k-3+|h|$. The $4$-cycle $(z,y,x,s)$ is not allowed to be full. Therefore $y \sim s$. Note that $d(y,h(y))= d(y,s) + d(s,h(y))=2k-2+|h|$. This yields contradiction with $d(y,h(y))=2k-1+|h|$.

%If $d(v,x)=d(x,w)=k$, let $u \in \gamma$ such that $x \sim u$ %%%(possibly with $u=w$). Note that $d(y,h(x))=d(u,h(x))=2k-%2+|h|$. By (TC), there exists $t \sim y,u$ such that %$d(t,h(x))=2k-3+|h|$. A full $4$-cycle $(y,x,u,t)$ is forbidden.
%If $y \sim u$, we get contradiction with the fact that $\gamma$ %is a geodesic. If $x \sim t$, we get $d(x,h(x))=d(x,t) + %d(t,h(x)) = 2k-2+|h|$. This yields contradiction with %$d(x,h(x))=2k-1+|h|$.

Case $H.2$. Let $|\gamma| = 2k+1, k \geq 1$. 

Note that the case $d(v,y)=k+1$ yields contradiction with the choice of the vertex $y$. Therefore $d(v,y)=k$.
Let $u \in \gamma$ such that $u \sim x$, $d(x,w)=d(u,w)+1$ (possibly with $u=w$). Note that $d(y,h(y)) = d(u,h(y)) = 2k-1+|h|$. Because $d(y,u)=2$, by (QC), there exists a vertex $s \sim y,u$ such that $d(s,h(y))=2k-2+|h|$. The $4-$cycle $(y,x,u,s)$ is not full. Therefore $x \sim s$. Then $d(x,h(y))=d(x,s)+d(s,h(y))=2k-1+|h|$. This yields contradiction with $d(x,h(y))=2k+|h|$.

Case $I.$ There exist vertices $c \in \beta$, $b \in \alpha$ such that $w \sim c \sim w'$; $h(v) \sim b \sim h(v')$, respectively. There do not exist vertices $a \in \alpha$, $e \in \beta$ such that $v \sim a \sim v'$; $h(w) \sim e \sim h(w')$, respectively. This case is similar to case $H$ for $v=w$, $h(v)=h(w)$.

\end{proof}

\begin{lemma}\label{3.3}
Let $h$ be a (simplicial) isometry with no fixed simplices of a simplicial complex $X$. Let $Y = \rm{Min}_{X}(h)$. Then $\rm{Min}_{X}(h) = \rm{Min}_{Y}(h)$.
\end{lemma}

\begin{proof}
Let $x \in X$ such that $d_{X}(x,h(x)) = |h|$. Then $x \in Y$. Let $y = h(x) \in Y$ such that $d_{Y}(y,h(y)) = |h|$. So $y \in \rm{Min}_{Y}(h)$ and therefore $Y \subset \rm{Min}_{Y}(h)$. Arguing similarly, we get $\rm{Min}_{Y}(h) \subset Y$. Hence $\rm{Min}_{X}(h) = \rm{Min}_{Y}(h)$.

\end{proof}

In the next theorem we show that the minimal displacement set in a weakly systolic complex is a systolic subcomplex.

\begin{theorem}\label{3.4}
Let $h$ be a (simplicial) isometry with no fixed simplices of a weakly systolic complex $X$. Then
the subcomplex $\rm{Min}_{X}(h)$ is systolic.
\end{theorem}

\begin{proof}
Let $Y = \rm{Min}_{X}(h)$.
Lemma \ref{3.3} implies that $Y = \rm{Min}_{Y}(h)$. 

The proof is by contradiction. Suppose $\rm{Min}_{X}(h)$ is not systolic.
Then, for $z \in \rm{Min}_{X}(h)$, 
there exists a full $5$-cycle $(x_{1}, ..., x_{5}) \subset [\rm{Min}_{X}(h)]_{z}$. According to Lemma \ref{3.3}, there exist vertices $y_{i} \in \rm{Min}_{X}(h)$ such that $h(x_{i})=y_{i}, d(h(x_{i}),x_{i}) = |h|, 1 \leq i \leq 5$. We have $d(h(z),z)= |h|$. Note that $(h(z);h(x_{1}), ..., h(x_{5})) \subset \rm{Min}_{X}(h)$.

Since $x_{i} \sim x_{i+1}$, $1 \leq i \leq 4$, $z \sim x_{i}$, $1 \leq i \leq 5$ and since $h$ is an isometry, we have $h(x_{i}) \sim h(x_{i+1})$, $1 \leq i \leq 4$, $h(z) \sim h(x_{i})$, $1 \leq i \leq 5$. Let $\gamma_{i}$ be a geodesic joining $x_{i}$ to $h(x_{i})$, $1 \leq i \leq 2$.
 Let $\gamma_{3}$ be a geodesic joining $z$ to $h(z)$.  Lemma \ref{3.1} implies that $\gamma_{i} \subset \rm{Min}_{X}(h)$, $1 \leq i \leq 3$.
 Note that \begin{center}$d(h(x_{i}),z)=|\gamma_{i}| - d(x_{i},z)=d(h(x_{i}),x_{i}) - d(x_{i},z) =|h|-1, 1 \leq i \leq 2$.\end{center} Then, by (TC), there exists a vertex $a \sim h(x_{1}), h(x_{2})$ such that $d(a,z)=|h|-2$.

%Let $\gamma_{i}$ be a geodesic from $x_{i}$ to $h(x_{i}), 1 %\leq i \leq 2$. Due to Lemma \ref{3.1}, $\gamma_{i} \subset %Min_{X}(h), 1 \leq i \leq 2$. 

\begin{figure}[ht]
   \centering
    \includegraphics[width=0.7\textwidth]{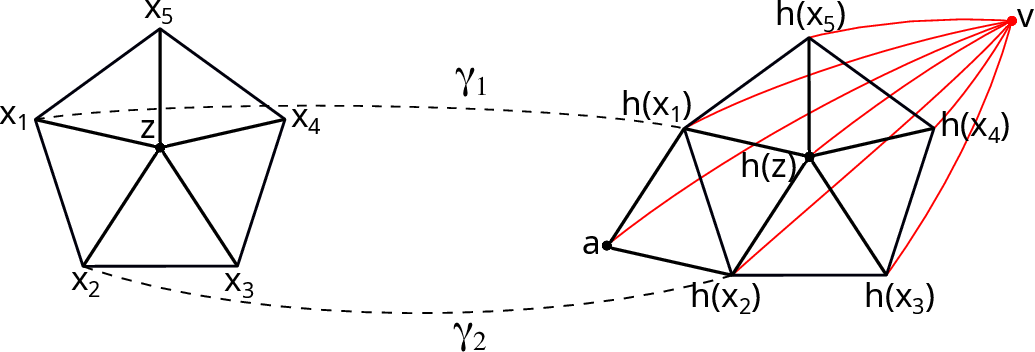}
  %  \caption{}
%    \label{Min_{X}(h) is systolic}
\end{figure}

We show that $\widehat{W}_{5} = (h(z);h(x_{1}), ..., h(x_{5});a)$ is an extended $5$-wheel belonging to $\rm{Min}_{X}(h)$.

$\bullet$ Suppose $h(x_{1}) \sim h(x_{3})$. Then $d(x_{1},x_{3}) = d(h(x_{1}),h(x_{3})) = 1$. On the other hand, since $(x_{1}, ..., x_{5})$ is a full cycle, $d(x_{1},x_{3})=2$. This implies a contradiction. Hence $h(x_{1}) \nsim h(x_{3})$. Arguing similarly, it follows that $(h(x_{1}), ..., h(x_{5}))$ is a full cycle.

$\bullet$ Suppose $a \sim h(z)$. Then $d(z,h(z))=d(z,a)+d(a,h(z))=|h|-1$. This yields contradiction with $d(z,h(z))=|h|$ which holds since $z \in \rm{Min}_{X}(h)$.

$\bullet$ Suppose $a \sim h(x_{3})$. The $4$-cycle $(a,h(x_{3}),h(z),h(x_{1}))$ is not allowed to be full. If either $a \sim h(z)$ or $h(x_{1}) \sim h(x_{3})$, we get a contradiction as shown above. 

$\bullet$ Suppose $a \sim h(x_{4})$. The $4$-cycle $(a,h(x_{4}),h(x_{3}),h(x_{2}))$ is not allowed to be full. If $a \sim h(x_{3})$, we get contradiction as shown above. If $h(x_{2}) \sim h(x_{4})$, we get contradiction since $(h(x_{1}), ..., h(x_{5}))$ is a full cycle.

Since $a \nsim h(z)$, we get $d(a,h(z))=d(a,h(x_{1})) + d(h(x_{1}),h(z)) = 2$. Therefore $d(z,h(z))=|h|=|h| - 2 + 2 = d(z,a) + d(a,h(z))$. Hence $a \in \gamma_{3}$. Since $\gamma_{3} \subset \rm{Min}_{X}(h)$, we have $a \in \rm{Min}_{X}(h)$.
In conclusion $\widehat{W}_{5}$ is an extended $5$-wheel belonging to $\rm{Min}_{X}(h)$. 

Because $X$ is weakly systolic, there exists $v \in X^{(0)}$ such that $\widehat{W}_{5} \subset X_{v}$. Hence $d(a,h(x_{4})) = d(a,v)+d(v,h(x_{4})) = 2$. On the other hand, since $\widehat{W}_{5} \subset \rm{Min}_{X}(h)$, we have $d(a,h(x_{4})) = d(a,h(x_{2})) + d(h(x_{2}),h(x_{3})) + d(h(x_{3}),h(x_{4}))= 3$. This yields a contradiction. So, for $z \in \rm{Min}_{X}(h)$, there exists no full $5$-cycle $(x_{1}, ..., x_{5}) \subset [\rm{Min}_{X}(h)]_{z}$. Hence $\rm{Min}_{X}(h)$ is systolic.

\end{proof}

%\begin{theorem}\label{3.5}
%Let $h$ be a (simplicial) isometry of a weakly systolic complex $X$ %having no fix-points such that
%the set $\rm{Min}(h) \subset X$ is isometrically embedded and it %satisfies the extended $5$-wheel condition (i.e. $(C_{5},W_{5})$, %$\hat{W}_{5}$-condition). Then $\rm{Min}(h) \subset X$ is weakly %systolic.
%\end{theorem}

%\begin{proof}

%There are no $5$-wheels by hypothesis. There are no full $4$-cycles %by weak systolicity. 

%\end{proof}

The following results on weakly systolic complexes are immediate consequences of the fact that the minimal displacement set of a hyperbolic isometry acting on such complex is a systolic subcomplex that embeds isometrically into the complex. Their systolic analogues, also given below, imply these consequences.

\begin{theorem}\label{3.5}
Let $h$ be a hyperbolic simplicial isometry of a uniformly
locally finite systolic complex $X$. Then in $X$ there is an $h^{n}$-invariant geodesic for some $n \geq 1$.
\end{theorem}

For the proof see \cite{E2}, Theorem $3.5$, page $46$.

\begin{corollary}\label{3.6}
Let $h$ be a hyperbolic simplicial isometry of a uniformly
locally finite weakly systolic complex $X$. Then in $X$ there is an $h^{n}$-invariant geodesic for some $n \geq 1$.
\end{corollary}

\begin{proof}

Let $Y = \rm{Min}_{X} (h)$. Theorem \ref{3.4} implies that $Y$ is systolic. Then, by Theorem \ref{3.5}, in $Y$ there is an $h^{n}$-invariant geodesic $\gamma$ for some $n \geq 1$. Since, by Theorem \ref{3.2}, $Y^{(1)}$ is isometrically embedded into $X$, the $h^{n}$-invariant geodesic $\gamma$ in $Y$, also belongs to $X$. This completes the proof.

\end{proof}

\begin{theorem}\label{3.7}
Let $h$ be a simplicial isometry of a uniformly locally finite
systolic complex $X$. Then in $X$ either there is an $h$-invariant simplex (elliptic case)
or there is an $h$-invariant thick geodesic (hyperbolic case).
\end{theorem}

For the proof see \cite{E2}, Theorem $3.8$, page $49$.

\begin{corollary}\label{3.8}
Let $h$ be a simplicial isometry of a uniformly locally finite weakly systolic complex $X$. Then in $X$ either there is an $h$-invariant simplex (elliptic case)
or there is an $h$-invariant thick geodesic (hyperbolic case).
\end{corollary}

\begin{proof}

Let $Y = \rm{Min}_{X} (h)$. Theorem \ref{3.4} implies that $Y$ is systolic. Then, by Theorem \ref{3.7}, in $Y$ either there is an $h$-invariant simplex (elliptic case)
or there is an $h$-invariant thick geodesic (hyperbolic case). Since, by Theorem \ref{3.2}, $Y^{(1)}$ is isometrically embedded into $X$, the $h$-invariant simplex of $Y$, respectively the $h$-invariant thick geodesic of $Y$, also belongs to $X$. 

\end{proof}

\begin{theorem}\label{3.9}
Let $h$ be a hyperbolic simplicial isometry of a uniformly locally finite systolic complex $X$.
If in $X$ there exists an $h^{n}$-invariant geodesic for some $n \geq 1$, then for any vertex $x \in \rm{Min}_{X}(h^{n}) \subset X$, there exists an $h^{n}$-invariant geodesic passing through $x$.
\end{theorem}

For the proof see \cite{E2}, Remark page $48$.

\begin{corollary}\label{3.10}
Let $h$ be a hyperbolic simplicial isometry of a uniformly locally finite weakly systolic complex $X$.
If in $X$ there exists an $h^{n}$-invariant geodesic for some $n \geq 1$, then for any vertex $x \in \rm{Min}_{X}(h^{n}) \subset X$, there exists an $h^{n}$-invariant geodesic passing through $x$.
\end{corollary}

\begin{proof}
Let $Y = \rm{Min}_{X}(h)$. Theorem \ref{3.4} implies that $Y$ is systolic.
According to Theorem \ref{3.5}, in $Y$ (and then, by Corollary \ref{3.6}, also in $X$) there exists an $h^{n}$-invariant geodesic for some $n \geq 1$.  Hence, by Theorem \ref{3.9}, for any vertex $x \in \rm{Min}_{X}(h^{n}) \subset Y$, there exists an $h^{n}$-invariant geodesic passing through $x$. Since, by Theorem \ref{3.2}, $Y^{(1)}$ is isometrically embedded into $X$, this implies that for any vertex $x \in \rm{Min}_{X}(h^{n}) \subset Y \subset X$, there exists an  $h^{n}$-invariant geodesic passing through $x$.
\end{proof}

\end{document}